\documentclass[12pt,leqno]{article}
\usepackage{amsfonts}
\usepackage{amsmath, amsthm, amsfonts, amssymb, color}
\usepackage{mathrsfs}

\usepackage[notcite,notref]{showkeys}

\newtheorem{thm}{Theorem}[section]

\newtheorem{lem}[thm]{Lemma}

\theoremstyle{definition}
\newtheorem{defn}{Definition}[section]
\newcommand{\scr}[1]{\mathscr #1}
\definecolor{wco}{rgb}{0.5,0.2,0.3}

\numberwithin{equation}{section} \theoremstyle{remark}
\newtheorem{rem}{Remark}[section]

\def\R{\mathbb R}  \def\ff{\frac} \def\ss{\sqrt} \def\B{\mathbf
B}
\def\N{\mathbb N} \def\kk{\kappa} 
\def\dd{\delta}  \def\vv{\varepsilon} \def\rr{\rho}
\def\<{\langle} \def\>{\rangle} \def\GG{\Gamma} \def\gg{\gamma}
    
\def\d{\text{\rm{d}}} \def\bb{\beta} \def\aa{\alpha} \def\D{\scr D}
  \def\si{\sigma} 
 \def\beq{\begin{equation}}  \def\F{\scr F}
 
\def\e{\text{\rm{e}}}  \def\OO{\Omega}  
 \def\tt{\tilde} 
 \def\P{\mathbb P} 
\def\C{\scr C}

\def\E{\mathbb E} 
  \def\LL{\Lambda}
 \def\B{\scr B}  
\def\to{\rightarrow}\def\ll{\lambda}
\def\8{\infty} \def\Y{\mathbb{Y}} 
\def\3{\triangle}\def\1{\lesssim}

\renewcommand{\tilde}{\widetilde}

\newcommand{\barray}{\begin{array}{ll}}
\newcommand{\earray}{\end{array}}

\newcommand{\wdt}{\widetilde}
\newcommand{\bea}{\begin{displaymath}\begin{array}{rl}}
\newcommand{\eea}{\end{array}\end{displaymath}}

\title{Exponential Mixing for  Retarded Stochastic Differential Equations}
\author{Jianhai Bao,\thanks{Department of Mathematics, Central South University,
Changsha, Hunan, 410075, P.R. China, majb@swansea.ac.uk} \and George
Yin,\thanks{Department of Mathematics, Wayne State University,
Detroit, MI 48202, USA, gyin@math.wayne.edu} \and Le Yi Wang,\thanks{Department of Electrical and
Computer Engineering, Wayne State University, Detroit,
MI 48202, USA, lywang@wayne.edu}  \and Chenggui
Yuan\thanks{Department of Mathematics, Swansea University, Singleton
Park, SA2 8PP, UK, C.Yuan@swansea.ac.uk} }

\begin{document}

\maketitle

\begin{abstract}
In this paper, we discuss exponential mixing property for Markovian
semigroups generated by   segment processes associated with several
class of retarded Stochastic Differential Equations (SDEs) which
cover SDEs with constant/variable/distributed time-lags. In
particular, we investigate the exponential mixing property for (a)
non-autonomous retarded SDEs by the  Arzel\`{a}--Ascoli tightness
characterization of the space $\C$ equipped with the uniform
topology  (b) neutral SDEs with continuous sample paths by a
generalized Razumikhin-type argument and a stability-in-distribution
approach and  (c) jump-diffusion retarded SDEs  by the Kurtz
criterion of tightness for the space $\D$ endowed with the Skorohod
topology.

\vskip 0.2 true in \noindent {\bf Keywords:} retarded stochastic
differential equation, invariant measure, exponential mixing, unform
metric, Skorohod metric

\vskip 0.2 true in \noindent
 {\bf AMS Subject Classification:}\    60H15, 60J25, 60H30, 39B82

 \end{abstract}

\section{Introduction}
Ergodic property of stochastic dynamical systems, which are
independent of the past history, has attracted lots of attentions.
For stochastic dynamical systems driven by continuous  noise
processes (e.g. Wiener process and fractional Brownian motion), we
refer to \cite{H05,MSH,O08,Rey,V97,Z09} and references cited
therein. And there has been intense interest in studying dynamical
systems subject to discontinuous Markov processes due to their
importance both in theory and in applications.   There is also
extensive literature on ergodicity of SDEs driven by L\'{e}vy
processes (e.g. Poisson process, $\aa$-stable process, cylindrical
$\aa$-stable process and subordinate Brownian motion), see e.g.
\cite{DXZ,K09,PSXZ,W12}, to name a few.

Many physical phenomena should be and in fact have already been
successfully modeled by stochastic dynamical systems whose evolution
in time is governed by random forces as well as intrinsic dependence
of  the state on a finite part of its past history. Such models may
be identified as retarded (functional) SDEs (see e.g. the monograph
\cite{M84} for more details). Relative to SDEs without memory, the
long-term behavior of retarded SDEs is not yet complete. There is a
few of literature on investigation in existence of stationary
solutions, see e.g. It\^o and Nisio
 \cite{IN} for retarded SDEs with infinite memory by
 the Prohorov-Skorohod theory of the totally bounded sets of stochastic
 processes, Bakhtin and  Mattingly \cite{BM} for retarded SDEs with
 additive noise and infinite memory by a Lyapunov function
 approach,   Liu \cite{Liu} and Rei$\beta$ et al. \cite{RR}
for infinite-dimensional retarded Langevin equations and
finite-dimensional semi-linear retarded SDEs driven by jump
processes with the diffusion term
 being
 independent of the past history,
 respectively, by the variation of constants
formula. For existence of an invariant measure, Es-Sarhir et al.
\cite{ESV} and Kinnally and Williams \cite{KW} considered retarded
SDEs with super-linear drift term and positivity constraints,
respectively; Bo and Yuan \cite{BY} investigated   reflected SDEs
with jumps and point delays.  With regard to uniqueness of invariant
measures,  by an asymptotic coupling method, Hairer et al.
\cite{HMS} addressed the open problem of uniqueness of invariant
measure for non-degenerate retarded SDEs under some appropriate
assumptions which need not guarantee existence of an invariant
measure, and Scheutzow \cite{S12} discussed a very simple linear
retarded SDE without the drift term.

 In this paper, we shall investigate the exponential
mixing property for several class of retarded SDEs which cover SDEs
with constant/variable/distributed delays. The content of this paper
is organized as follows.   Section \ref{sec:2} discusses the
exponential mixing property for a class of non-autonomous retarded
SDEs by the  Arzel\`{a}--Ascoli tightness characterization of the
space $\C$ equipped with the uniform topology. In Section
\ref{sec:3}, we proceed to neutral SDEs with continuous sample paths
by a generalized Razumikhin-type argument and a
stability-in-distribution approach. The last section is devoted to
the exponential ergodicity for jump-diffusion retarded SDEs  by the
Kurtz criterion of tightness for the space $\D$ endowed with the
Skorohod topology.

\section{Exponential Mixing for Retarded SDEs}\label{sec:2}
We start with some notation. For each integer $n\ge1$, let
$(\R^n,\<\cdot,\cdot\>,|\cdot|)$ be the $n$-dimensional Euclidean
space and  $\R^n\otimes\R^m$ denote the totality of all $n\times m$
matrices  endowed with the Frobenius norm $\|\cdot\|$.  For a fixed
constant $\tau>0$, $\C:=C([-\tau,0];\R^n)$ stands for the family of
all continuous mappings $\zeta:[-\tau,0]\mapsto\R^n$ equipped with
the uniform norm
$\|\zeta\|_\8:=\sup_{-\tau\le\theta\le0}|\zeta(\theta)|$.   For any
continuous function $f:[-\tau,\8)\mapsto\R^n$ and $t\ge0$, let
$f_t\in\C$ be such that $f_t(\theta)=f(t+\theta)$ for each
$\theta\in[-\tau,0]$. As usual,  $\{f_t\}_{t\ge0}$ is called the
segment process of $\{f(t)\}_{t\ge-\tau}$. Let $W(t)$ be an
$m$-dimensional Wiener process defined on a complete filtered
probability space $(\OO,\F,\{\F_t\}_{t\ge0},\P)$. The
 notation $\mathcal {P}(\C)$ denotes the collection of all probability
measures on $(\C,\B(\C))$, $\B_b(\C)$ means the set of all bounded
measurable functions $F:\C\to\R$ endowed with the uniform norm
$\|F\|_0:=\sup_{\phi\in\C}|F(\phi)|$, and   $\mu(\cdot)$ stands for
a probability measure on $[-\tau,0].$ For any $F\in\B_b(\C)$ and
$\pi(\cdot)\in\mathcal {P}(\C)$, let $\pi(F):=\int_\C
F(\phi)\pi(\d\phi)$. Throughout this paper, $c>0$ is a generic
constant whose value may change from line to line but independent of
the time parameters.

In this section, we consider a retarded SDE on
$(\R^n,\<\cdot,\cdot\>,|\cdot|)$ in the framework
\begin{equation}\label{A1}
\d X(t)=b(t,X_t)\d t+\si(t,X_t)\d W(t),\ \ \ t>0
\end{equation}
with the initial
data $X_0=\xi\in\C$, where $b:[0,\8)\times\C\mapsto\R^n$ and
$\si:[0,\8)\times\C\mapsto\R^{n}\otimes\R^m$ are  measurable and
locally Lipschitz with respect to the second variable. Throughout
this section, we assume that the initial value $\xi\in\C$ is
independent of $\{W(t)\}_{t\ge0}.$

For any $\phi,\psi\in\C$ and  $t,p\ge0$, we assume that
\begin{enumerate}
\item[(H$1$)] There exist $\aa_1>\aa_2>0$ such that
\begin{equation*}
\begin{split}
&\E\{|\phi(0)-\psi(0)|^p(2\<\phi(0)-\psi(0),b(t,\phi)-b(t,
\psi)\> \quad +\|\si(t,\phi)-\si(t,\psi)\|^2)\}\\
&\quad\le-\aa_1\E|\phi(0)-\psi(0)|^{2+p}+\aa_2\sup_{-\tau\le\theta\le0}\E\{|\phi(0)-\psi(0)|^p|\phi(\theta)-\psi(\theta)|^2\};
\end{split}
\end{equation*}
\item[(H$2$)] There exists $\aa_3>0$ such that
\begin{equation*}
\E\|\si(t,\phi)-\si(t,\psi)\|^{2+p}\le
\aa_3\sup_{-\tau\le\theta\le0}\E(|\phi(\theta)-\psi(\theta)|^{2+p}).
\end{equation*}
\end{enumerate}

The following remark shows that there are some examples such that
(H1) and (H2).

\begin{rem}
{\rm Let  $b(t,\phi)=b(t,\phi(0),\phi(-\dd(t)))$ and
$\si(t,\phi)=\si(t,\phi(0),\phi(-\dd(t)))$ with $\phi\in\C$, where
$\dd:[0,\8)\mapsto[0,\tau]$ is a measurable function. For any
$\phi\in\C$ and $t\ge0$, if
\begin{equation*}
\begin{split}
&2\<\phi(0)-\psi(0),b(t,\phi(0),\phi(-\dd(t)))-b(t,\psi(0),\psi(-\dd(t)))\>\\
&\qquad\quad +\|\si(t,\phi(0),\phi(-\dd(t)))-\si(t,\psi(0),\psi(-\dd(t)))\|^2\\
&\quad\le-\aa_1|\phi(0)-\psi(0)|^2+\aa_2|\phi(-\dd(t))-\psi(-\dd(t))|^2,
\end{split}
\end{equation*}
and
\begin{equation*}
\|\si(t,\phi)-\si(t,\psi)\|^2\le
\aa_3(|\phi(0)-\psi(0)|^2+|\phi(-\dd(t))-\psi(-\dd(t))|^2),
\end{equation*}
then (H1) and (H2) hold respectively for some appropriate constants
$\aa_1,\aa_2,\aa_3>0$. On the other hand, for arbitrary $\phi\in\C$
and $t\ge0$, if
\begin{equation*}
\begin{split}
&2\<\phi(0)-\psi(0),b(t,\phi)-b(t,\psi)\>+\|\si(t,\phi)-\si(t,\psi)\|^2\\
&\quad\le-\aa_1|\phi(0)-\psi(0)|^2+\aa_2\int_{-\tau}^0|\phi(\theta)-\psi(\theta)|^2\mu(\d\theta),
\end{split}
\end{equation*}
and
\begin{equation*}
\|\si(t,\phi)-\si(t,\psi)\|^2\le
\aa_3\Big(|\phi(0)-\psi(0)|^2+\int_{-\tau}^0|\phi(\theta)-\psi(\theta)|^2\mu(\d\theta)\Big),
\end{equation*}
where $\mu(\cdot)$ is a probability measure on $[-\tau,0]$, then
(H1) and (H2) are also fulfilled for some $\aa_1,\aa_2,\aa_3>0$.
From the previous discussions, we deduce that our framework  cover
SDEs with  constant/variable/distributed delays. }
\end{rem}

 Since $b$ and $\si$ are locally
Lipschitz, \eqref{A1} admits a unique local strong solution
$\{X(t,\xi)\}_{t\ge-\tau}$ with the initial value $\xi\in\C$.
Moreover, (H1) guarantees that $\P(\rr=\8)=1$, where
$\rr:=\lim_{n\to\8}\rr_n$ is
 the life time of $\{X(t,\xi)\}_{t\ge-\tau}$ with $\rr_n:=\inf\{t>0,|X(t)|\ge n\}$ for any integer $n\ge0.$
Therefore, \eqref{A1} has a unique strong solution
$\{X(t,\xi)\}_{t\ge-\tau}$ under (H1).

\smallskip

\smallskip

For citation convenience, several fundamental inequalities   are
summarized in the following lemmas.

\begin{lem}\label{Inq1}
{\rm (\cite[Lemma 8.1]{IN}) Let  $u,v:[0,\8)\mapsto\R_+$ be
continuous functions and $\bb>0$. If
\begin{equation*}
u(t)\le u(s)-\bb\int_s^tu(r)\d r+\int_s^tv(r)\d r,\ \ \ 0\le s<t<\8,
\end{equation*}
then
\begin{equation*}
u(t)\le u(0)+\int_0^t\e^{-\bb(t-r)}v(r)\d r.
\end{equation*}
}
\end{lem}

\begin{lem}\label{Inq2}
{\rm (\cite[Lemma 8.2]{IN}) Let $u:[0,\8)\mapsto\R_+$ be a
continuous function and $\dd>0, \aa>\bb>0.$ If
\begin{equation*}
u(t)\le\dd+\bb\int_0^t\e^{-\aa(t-s)}u(s)\d s,\ \ \ t\ge0,
\end{equation*}
then $u(t)\le(\dd\aa)/(\aa-\bb)$. }
\end{lem}

\begin{lem}\label{Halanay}
{\rm (\cite[Theorem 2.1]{MG})For $a,b>0$, let $u(\cdot)$ be a
nonnegative function such that
\begin{equation*}
u'(t)\le -a u(t)+b\sup_{t-\tau\le s\le t} u(s),\ \ \ t>0
\end{equation*}
and $u(s)=|\psi(s)|$ is continuous for $s\in[-\tau,0]$. Then, for
$a>b>0,$ there exists $\ll>0$ such that
\begin{equation*}
u(t)\le \Big(\sup_{-\tau\le s\le 0}u(s)\Big)\e^{-\ll t},\ \ \ t\ge
0.
\end{equation*}
}
\end{lem}

With  Lemma \ref{Inq1} and Lemma \ref{Inq2} in hand, we can obtain a
uniform bound of segment process $\{X_t(\xi)\}_{t\ge-\tau}$ with the
initial data $\xi\in\C$,  which plays a crucial role in
investigation on existence of an invariant measure of \eqref{A1}.
For notation brevity, in the sequel we shall write $X(t)$ and $X_t$
instead of $X(t,\xi)$ and $X_t(\xi)$ respectively.

\begin{lem}\label{uni}
{\rm Assume that (H1) and (H2) hold. Then there exists a
sufficiently small $\kk>0$ such that
\begin{equation}\label{A10}
\sup_{t\ge-\tau}\E\|X_t(\xi)\|_\8^{2+\kk}<\8.
\end{equation}

}
\end{lem}

\begin{proof}
 For any $\kk>0$,  by the It\^o formula,  we obtain that
\begin{equation}\label{A2}
\begin{split}
\rr(t)&:=\E|X(t)|^{2+\kk}\\
&\le\ff{2+\kk}{2}\E\int_0^t|X(s)|^\kk\{2\<X(s),b(s,X_s)\>+\|\si(s,X_s)\|^2\}\d
s\\
&\quad+|\xi(0)|^{2+\kk}+\ff{\kk(2+\kk)}{2}\E\int_0^t|X(s)|^\kk\cdot\|\si(s,X_s)\|^2\d
s\\
&=:I_1(t)+I_2(t).
\end{split}
\end{equation}
By (H1) and (H2), it is readily to see that there exist
$\nu_1>\nu_2>0$ such that
\begin{equation}\label{A3}
\begin{split}
\E\{|\phi(0)|^\kk(2\<\phi(0),b(t,\phi)\>+\|\si(t,\phi)\|^2) \}&\le
-\nu_1\E|\phi(0)|^{2+\kk}\\
&\quad+\nu_2\sup_{-\tau\le\theta\le0}\E(|\phi(0)|^\kk\cdot|\phi(\theta)|^2)+c
\end{split}
\end{equation}
for any $t\ge0$ and $ \phi\in\C.$ This, together with the Young
inequality:
\begin{equation}\label{A15}
a^\bb b^{1-\bb}\le \bb a+(1-\bb)b,\ \ \ a,b>0,\bb\in(0,1),
\end{equation}
gives that
\begin{equation*}
\begin{split}
I_1(t)
&\le\ff{2+\kk}{2}\int_0^t\{-\nu_1\rr(s)+\nu_2\sup_{-\tau\le\theta\le0}\E(|X(s)|^\kk\cdot|X(s+\theta)|^2)+c\}\d
s\\
&\le-\ff{(2+\kk)\nu_1}{2}\int_0^t\rr(s)\d
s+\ff{(2+\kk)\nu_2}{2}\int_0^t\Big\{\ff{\kk}{2+\kk}\rr(s)+\ff{2}{2+\kk}\sup_{-\tau\le\theta\le
s}\rr(r)+c\Big\}\d
s\\
&\le-\ff{(2+\kk)}{2}\Big(\nu_1-\ff{\nu_2\kk}{2+\kk}-\kk\Big)\int_0^t\rr(s)\d
s+\int_0^t \{c+\nu_2r(s) \}\d s,
\end{split}
\end{equation*}
in which $r(t):=\sup_{0\le s\le t}\rr(s)$. According  to (H2) and
  Young's inequality \eqref{A15}, it follows that
\begin{equation*}
\begin{split}
I_2(t)&\le\|\xi\|_\8^{2+\kk}+\ff{\kk(2+\kk)}{2}\int_0^t\Big\{\ff{\kk}{2+\kk}\rr(s)+\ff{2}{2+\kk}\E\|\si(s,X_s)\|^{2+\kk}\Big\}\d
s\\
&\le
\|\xi\|_\8^{2+\kk}+\ff{c\kk(2+\kk)}{2}\int_0^t\{1+\rr(s)+r(s)\}\d s.
\end{split}
\end{equation*}
Hence, we arrive at
\begin{equation}\label{A2}
\rr(t)\le\|\xi\|_\8^{2+\kk}-\ll_1\int_0^t\rr(s)\d
s+\int_0^t\{c+\ll_2r(s)\}\d s,
\end{equation}
where, for a sufficiently small $\kk\in(0,1)$,
\begin{equation*}
\ll_1:=\ff{(2+\kk)}{2}\Big(\nu_1-\ff{\nu_2\kk}{2+\kk}-(c+1)\kk\Big)>\ll_2:=\nu_2+\ff{c\kk(2+\kk)}{2}
\end{equation*}
due to $\nu_1>\nu_2$. Combining \eqref{A2}  with Lemma \ref{Inq1}
gives that
\begin{equation}\label{A14}
\rr(t)\le\|\xi\|_\8^{2+\kk}+\int_0^t\e^{-\ll_1(t-s)}\{c+\ll_2r(s)\}\d
s.
\end{equation}
For a nondecreasing function $u:[0,\8)\mapsto\R^+$ and any $\ll>0$,
observe that the integral
\begin{equation*}
\int_0^t\e^{-\ll(t-s)}u(s)\d s\ \  \mbox{ is nondecreasing with
respect to } t
\end{equation*}
due to the fact that
\begin{equation*}
\int_0^t\e^{-\ll(t-r)}u(s)\d s=\ff{(1-\e^{-\ll
t})u(0)}{\ll}+\int_0^t\ff{1-\e^{-\ll(t-s)}}{\ll}\d u(s).
\end{equation*}
By the nondecreasing property of $r(t)$ with respect to $t$, we
therefore infer from \eqref{A14} that
\begin{equation*}
r(t)\le\|\xi\|_\8^{2+\kk}+\int_0^t\e^{-\ll_1(t-s)}\{c+\ll_2r(s)\}\d
s\le c+\ll_2\int_0^t\e^{-\ll_1(t-s)}r(s)\d s.
\end{equation*}
Thanks to $\ll_1>\ll_2$,  Lemma \ref{Inq2}   leads to
\begin{equation}\label{A8}
\sup_{t\ge-\tau}\rr(t)<\8.
\end{equation}
Next, for any $t\ge\tau$, applying the It\^o formula, together with
the Burkhold-Davis-Gundy inequality and the Young inequality
\eqref{A15}, we deduce from \eqref{A3} that
\begin{equation*}
\begin{split}
\E\|X_t\|_\8^{2+\kk}&\le \rr(t-\tau)+c\int_{t-\tau}^t\{1+\rr(s)+r(s)\}\d s\\
&\quad+(2+\kk)\E\Big(\sup_{-\tau\le\theta\le0}\Big|\int_{t-\tau}^{t+\theta}|X(s)|^\kk\<X(s),\si(s,X_s)\d
W(s)\>\Big|\Big)\\
&\le
\ff{1}{2}\E\|X_t\|_\8^{2+\kk}+\rr(t-\tau)+c\int_{t-\tau}^t\{1+\rr(s)+r(s)\}\d
s.
\end{split}
\end{equation*}
That is,
\begin{equation}\label{A16}
\E\|X_t\|_\8^{2+\kk}\le2\rr(t-\tau)+c\int_{t-\tau}^t\{1+\rr(s)+r(s)\}\d
s,\ \ \ t\ge\tau.
\end{equation}
Moreover, note that
\begin{equation*}
\E\|X_t\|_\8^{2+\kk} \le\|\xi\|_\8^{2+\kk}+\E\Big(\sup_{0\le
t\le\tau}|X(t)|^{2+\kk}\Big),\ \ \ t\in[0,\tau].
\end{equation*}
Following a similar argument to derive \eqref{A16}, we deduce that
\begin{equation}\label{A17}
\E\|X_t\|_\8^{2+\kk}\le c+c\int_0^t\{1+\rr(s)+r(s)\}\d s,\ \ \
t\in[0,\tau].
\end{equation}
Then the desired assertion \eqref{A10}  follows by taking
\eqref{A8}-\eqref{A17} into account.
\end{proof}

\begin{defn}
{\rm A probability measure $\pi(\cdot)\in\mathcal {P}(\C)$ is called
an invariant measure of \eqref{A1} if, for arbitrary $F\in\B_b(\C)$,
\begin{equation*}
\pi(P_tF)=\pi(F),\ \ \ t\ge 0,
\end{equation*}
where $P_tF(\xi):=\E F(X_t(\xi))$.  }
\end{defn}

\begin{rem}
{\rm If $\pi(\cdot)\in\mathcal {P}(\C)$ is an invariant measure of
\eqref{A1} and the initial segment enjoys the same law, by
\cite[Lemma 1.1.9, p.14]{A09}, independence of $\xi\in\C$ and
$\{W(t)\}_{t\ge0}$ and smooth property of conditional expectation,
 one has
\begin{equation*}
\pi(F)=\int_\C\E
F(X_t(\eta))\pi(\d\eta)=\E(\E(F(X_t(\xi)))|\F_0)=\E(F(X_t(\xi))).
\end{equation*}
Then we conclude that $X_t(\xi)$ shares the law $\pi\in\mathcal
{P}(\C)$, i.e., the law of $X_t(\xi)$ is invariant under time
translation.

}
\end{rem}

The main result of this section is stated below.

\begin{thm}\label{Erg}
{\rm  Under (H1) and (H2), \eqref{A1} has a unique
 invariant measure $\pi(\cdot)\in\mathcal {P}(\C)$ and is
 exponentially mixing. More precisely, there exists $\ll>0$ such
 that
\begin{equation}\label{B6}
|P_tF(\xi)-\pi(F)|\le c\e^{-\ll t},\ \ \ t\ge0,\ F\in\B_b(\C),\
\xi\in\C.
\end{equation}
 }
\end{thm}

\begin{proof}
The whole proof of this theorem is divided into the following three
steps.
\smallskip

\noindent{\bf Step 1: Existence of an Invariant Measure.} The proof
on existence of an invariant measure  is due to the classical
Arzel\`{a}--Ascoli tightness characterization of the space $\C$.
Recall   that $X_t(\xi)$ admits by \cite[Theorem 1.1, p.51]{M84} the
Markovian property although the solution process $X(t,\xi)$ is not
Markovian. For arbitrary integer $n\ge1$, set
\begin{equation*}
\mu_n(\cdot):=\ff{1}{n}\int_0^n\P_t(\xi,\cdot)\d t,
\end{equation*}
where $\P_t(\xi,\cdot)$ is the Markovian transition kernel of
$X_t(\xi)$. By the Krylov-Bogoliubov theorem \cite[Theorem 3.1.1,
p.21]{DZ}, to show existence of an invariant measure, it is
sufficient to verify that  $\{\mu_n(\cdot)\}_{n\ge1}$ is relatively
compact. Note that the phase space $\C$ for the segment process
$X_t(\xi)$ is a complete separable space under the uniform metric
$\|\cdot\|_\8$ (see e.g. \cite[p.220]{B}). Taking
\cite[Theorem 6.2, p.37]{B} into consideration, we need only
show that $\{\mu_n(\cdot)\}_{n\ge1}$ is tight. Moreover, thanks to
\cite[Theorem 8.2, p.55]{B}, it suffices to claim that
\begin{equation}\label{e8}
\lim_{\dd\downarrow0}\sup_{n\ge1}\mu_n(\varphi\in\C:w_{[-\tau,0]}(\varphi,\dd)\ge\vv)=0
\end{equation}
for any $\vv>0$, where $w_{[-\tau,0]}(\varphi,\dd)$, the modulus of
continuity of  $\varphi\in\C$ (see e.g. \cite[p.54]{B}), is defined
by
\begin{equation*}
w_{[-\tau,0]}(\varphi,\dd):=\sup_{|s-t|\le\dd,s,t\in[-\tau,0]}|\varphi(s)-\varphi(t)|,\
\ \ \dd>0.
\end{equation*}
In the sequel, for simplicity we write $X(t)$ and $X_t$ instead of
$X(t,\xi)$ and $X_t(\xi)$ respectively. Since
\begin{equation*}
\begin{split}
I(t,\dd):&=\sup_{t\le v\le u\le t+\tau,0\le
u-v\le\dd}|X(u)-X(v)|\\
 &\le\sup_{t\le v\le u\le
t+\tau,0\le u-v\le\dd}\int_v^u|b(s,X_s)|\d s+\sup_{t\le v\le u\le
t+\tau,0\le
u-v\le\dd}\Big|\int_v^u\si(s,X_s)\d W(s)\Big|\\
&=:I_1(t,\dd)+I_2(t,\dd),\ \ \ t\ge\tau,
\end{split}
\end{equation*}
one has
\begin{equation*}
\begin{split}
\P(I(t,\dd)\ge \vv)&\le \P(I_1(t,\dd)\ge  \vv/2)+
\P(I_2(t,\dd)\ge\vv/2).
\end{split}
\end{equation*}
For any $\wdt \vv\in(0,1)$, by the Chebyshev inequality and Lemma
\ref{uni}, there exists an $R_0>0$ sufficiently large such that
\begin{equation}\label{B1}
\begin{split}
\P(&\|X_t\|_\8> R_0)+\P(\|X_{t+\tau}\|_\8> R_0)\le
R^{-2}_0\sup_{t\ge-\tau}(\E\|X_{t+\tau}\|_\8^2+\E\|X_t\|_\8^2)\le\wdt \vv.
\end{split}
\end{equation}
Moreover, since $b$ enjoys locally bounded property, there exists a
sufficiently small $\dd_0>0$ such that
\begin{equation}\label{B2}
\begin{split}
\P(I_1(t,\dd)\ge  \vv/2|\ \ \ \|X_t\|_\8\le R_0,\|X_{t+\tau}\|_\8\le
R_0)=0,\ \ \ \dd<\dd_0.
\end{split}
\end{equation}
 Accordingly, we obtain from \eqref{B1} and \eqref{B2} that
\begin{equation}\label{B3}
\begin{split}
\P(I_1(t,\dd)\ge \vv/2) &\le\P (I_1(t,\dd)\ge \vv/2 |\ \ \
\|X_t\|_\8\le R_0,\|X_{t+\tau}\|_\8\le
R_0)\\
&\quad+\P(\|X_t\|_\8\ge R_0)+\P(\|X_{t+\tau}\|_\8\ge
R_0)\\
&\le \wdt \vv.
\end{split}
\end{equation}
On the other hand, for $\kk\in(0,1)$ such that Lemma \ref{uni} holds
and arbitrary $0\le s\le t$, by the Burkhold-Davis-Gundy inequality,
(H2) and Lemma \ref{uni}, it follows that
\begin{equation*}
\begin{split}
\E\Big|\int_s^t\si(r,X_r)\d W(r)\Big|^{2+k} &\le
c(t-s)^{\kk/2}\int_s^t\{1+\E\|X_r\|_\8^{2+\kk}\}\d
r\\
&\le c(t-s)^{1+\kk/2}.
\end{split}
\end{equation*}
This, combining with  the Kolmogrov tightness criterion
\cite[Problem 4.11, p.64]{KS},  implies that
\begin{equation}\label{d1}
\lim_{\dd\downarrow0}\sup_{t\ge\tau}\P(I_2(t,\dd)\ge\vv/2)=0.
\end{equation}
Consequently, \eqref{e8} follows from \eqref{B3}, \eqref{d1},  the
arbitrariness of $\wdt \vv$, and by noticing that
\begin{equation*}
\begin{split}
\mu_n(\varphi\in\C:w_{[-\tau,0]}(\varphi,\dd)\ge\vv)
&\le\ff{2\tau}{n}+\ff{1}{n}\int_{\tau}^{n}\P(I(t,\dd)\ge\vv)\d t
\end{split}
\end{equation*}
for $n>\tau.$ Since $\{\mu_n(\cdot)\}_{n\ge1}$ is relative compact
due to \eqref{e8}, there exists a subsequence, still denoted by
$\{\mu_n(\cdot)\}_{n\ge1}$ without confusion, such that
$\{\mu_n(\cdot)\}_{n\ge1}$ converges weakly to some
$\pi(\cdot)\in\mathcal {P}(\C)$, which indeed is an invariant
measure of   $X_t(\xi)$ by \cite[Theorem 3.1.1, p.21]{DZ} and
recalling from  \cite[Theorem 3.1, p.67]{M84} that the
stochastically continuous semigroup $P_t$ is a Feller semigroup,
i.e., $P_tf\in C_b(\C)$ for any $F\in C_b(\C)$ and $t\ge0.$

\smallskip

\noindent{\bf Step 2: Uniqueness of Invariant Measures.} By the
It\^o formula, it is easy to see that
\begin{equation}\label{eq14}
\begin{split}
u(t):=\E&|X(t,\xi)-X(t,\eta)|^2\\&=|\xi(0)-\eta(0)|^2
+\int_0^t\E\{2\<X(s,\xi)-X(s,\eta),b(s,X_s(\xi))-b(s,X_s(\eta))\>\\
&\quad+\|\si(s,X_s(\xi))-\si(s,X_s(\eta))\|^2\}\d s.
\end{split}
\end{equation}
Differentiating with respect to $t$ on both sides of \eqref{eq14},
one has from (H1) with $p=0$ that
\begin{equation*}
 u'(t)\le-\aa_1u(t)+\aa_2\sup_{t-\tau\le s\le t}|u(s)|.
\end{equation*}
Then, Lemma \ref{Halanay} yields that
\begin{equation}\label{B4}
\E|X(t,\xi)-X(t,\eta)|^2\le \|\xi-\eta\|_\8^2\e^{-\ll t},\ \ \ t\ge0
\end{equation}
for some $\ll>0.$ Next, for any $t\ge\tau$, by the It\^o formula,
(H1) and the Burkhold-Davis-Gundy inequality, we arrive at
\begin{equation*}
\begin{split}
\E\|X_t(\xi)-X_t(\eta)\|^2_\8&\le\E|X(t-\tau,\xi)-X(t-\tau,\eta)|^2+c\E\int_{t-\tau}^t|X(s,\xi)-X(s,\eta)|^2\d
s\\
&\quad+c\int_{t-\tau}^t\sup_{s-\tau\le r\le
s}\E|X(r,\xi)-X(r,\eta)|^2\d
s+\ff{1}{2}\E\|X_t(\xi)-X_t(\eta)\|^2_\8.
\end{split}
\end{equation*}
This, in addition to \eqref{B4}, gives that
\begin{equation}\label{B5}
\begin{split}
\E\|X_t(\xi)-X_t(\eta)\|^2_\8&\le c\|\xi-\eta\|_\8^2\e^{-\ll
(t-\tau)}+c\sup_{t-2\tau\le s\le t}\E|X(s,\xi)-X(s,\eta)|^2\\
&\le c\e^{-\ll t},\ \ \ t\ge\tau.
\end{split}
\end{equation}
Observe that \eqref{B5} still holds for $t\in[0,\tau].$ On the basis
of \eqref{B5}, we claim that $\pi(\cdot)\in\mathcal {P}(\C)$ is the
unique invariant measure. Indeed, if $\pi^\prime(\cdot)\in\mathcal
{P}(\C)$ is also an invariant measure, for any bounded Lipschitz
function $f:\C\mapsto\R$, by \eqref{B5}  and the invariance of
$\pi(\cdot),\pi^\prime(\cdot)\in\mathcal {P}(\C)$, it follows that
\begin{equation*}
|\pi(f)-\pi^\prime(f)|\le\int_{\C\times\C}| P_tf(\xi)-
P_tf(\eta)|\pi(\d\xi)\pi^\prime(\d\eta)\le c\e^{-\ll t},\ \ \ \
t\ge0.
\end{equation*}
As a result, one gets the uniqueness of invariant measures  by
taking $t\to\8$ and applying \cite[Proposition 2.2, p.3]{IW89} and
\cite[Lemma 7.1.5, p.125]{DZ} for any $f\in C_b(\C)$, the set of all
bounded, continuous real-valued functions on $\C.$

\noindent{\bf Step 3: Exponential Mixing.} By the invariance of
$\pi\in\mathcal {P}(\C)$, for any $F\in\B_b(\C)$, it follows that
\begin{equation*}
 |P_tF(\xi)-\pi(F)|\le\int_\C| P_tF(\xi)-P_tF(\eta)|\pi(\d\eta).
\end{equation*}
Thus, the desired assertion \eqref{B6} follows by taking \eqref{B5}
and \cite[Lemma 7.1.5, p.125]{DZ} into consideration.
\end{proof}

\begin{rem}
{\rm Let $\ll>0$  be an appropriate constant and $V:\R^n\mapsto\R_+$
a $C^2$-function. By applying the It\^o formula to $\e^{\ll t}V(x)$,
under some appropriate conditions, Es-Sarhir et al.
\cite[Proposition 2.1]{ESV} and Bo and Yuan \cite[Proposition]{BY}
showed the uniform boundedness of segment processes, where $\ll>0$
need to be sufficiently large in the argument of \cite[Proposition
2.1]{ESV}. Based on the uniform boundedness of segment processes,
they also discussed existence of an invariant measure for the
corresponding Markovian transition semigroup. Although the method
adopted therein applies to SDEs with
  constant delay, it seems not to work for the case of
variable time-lag since the differentiable property of delay
function is not available. While,   Theorem \ref{Erg} covers a wide
range of retarded SDEs  which include non-autonomous SDEs with
constant/variable/distributed delays as their special cases. }
\end{rem}

\section{Exponential Mixing for Neutral  SDEs}\label{sec:3}
In the previous section, we discuss the exponential ergodicity of
the Markovian transition semigroups generated by the associated
segment processes for a class of non-autonomous retarded SDEs. In
this section, we proceed to discuss the exponential mixing property
for another class of stochastic equation depending on past and
present values but that involves derivatives with delays as well as
the function itself. Such equations historically have been called
neutral SDEs, which have many applications in variational problems,
chemical engineering system and optimal stochastic control (see e.g.
\cite[Chapter 6]{M08}). By a close inspection of the argument of
Theorem \ref{Erg}, we note that the Arzel\`{a}--Ascoli method
adopted in Theorem \ref{Erg} seems hard to apply to  neutral SDEs
although it can deal with the non-autonomous cases. In this section,
we shall put forward another method, called
stability-in-distribution approach, to cope with the exponential
mixing for the neutral SDEs.

Consider a neutral SDE on $\R^n$
\begin{equation}\label{C1}
\d\{X(t)-G(X_t)\}=b(X_t)\d t+\si(X_t)\d W(t)
\end{equation}
with the initial value $X_0=\xi\in\C$ which is independent of
$\{W(t)\}_{t\ge0}$, where $G:\C\mapsto\R^n$ is measurable and
continuous such that $G(0)=0$, and $b:\C\mapsto\R^n,
\si:\C\mapsto\R^{n}\otimes\R^m$ are measurable and locally
Lipschitz.

For any $\phi,\psi\in\C$, we assume that
\begin{enumerate}
\item[(A$1$)] There exists $\kk\in(0,1)$ such that
$\E|G(\phi)-G(\psi)|\le\kk\sup_{-\tau\le\theta\le
0}\E|\phi(\theta)-\psi(\theta)|^2$.

\item[(A$2$)]

There exist $\aa_1>\aa_2>0$ such that
\begin{equation*}
\begin{split}
&\E\{2\<\phi(0)-\psi(0)-(G(\phi)-G(\psi)),b(\phi)-b(
\psi)\>+\|\si(\phi)-\si_2(\psi)\|^2\}\\
&\quad\le-\aa_1\E|\phi(0)-\psi(0)|^2+\aa_2\sup_{-\tau\le\theta\le0}\E|\phi(\theta)-\psi(\theta)|^2.
\end{split}
\end{equation*}
\item[(A$3$)] There exists $\aa_3>0$ such that
\begin{equation*}
\E\|\si(\phi)-\si(\psi)\|^2\le
\aa_3\sup_{-\tau\le\theta\le0}\E|\phi(\theta)-\psi(\theta)|^2.
\end{equation*}
\end{enumerate}

Under (A1)-(A2), \eqref{C1} has a unique strong solution
$\{X(t,\xi)\}_{t\ge0}$ with the initial data $\xi\in\C.$ Before the
statement of our main result,  we first provide a generalized
Razumikhin-type theorem (see e.g. \cite[Theorem 6.1, p.221]{M08})
which guarantees that the segment process admits a uniform bound
although the  equation \eqref{C1} need not admit an equilibrium.

\begin{lem}\label{Raz}
{\rm Let (A1) hold and assume further that there exist
$\dd\ge0,\ll>0$ such that
\begin{equation}\label{eq1}
\begin{split}
\E\{2\<\phi(0)-G(\phi),&b(\phi)\>+\|\si(\phi)\|^2\}\le\dd-\ll\E|\phi(0)-G(\phi)|^2
\end{split}
\end{equation}
provided that, for some $q>(1-\kk)^{-2}$,
\begin{equation}\label{eq2}
\E|\phi(\theta)|^2<q|\phi(0)-G(\phi)|^2,\ \ \ -\tau\le\theta\le0.
\end{equation}
Then there exists $\gg<\ll$  sufficiently small such that
\begin{equation}\label{C3}
\E|X(t)|^2\le \ff{\dd/\ll+\e^{-\gg
t}(1+\kk)^2\|\xi\|_\8^2}{(1-\kk\e^{\gg\tau/2})^2},\ \ \ t\ge-\tau.
\end{equation}
}
\end{lem}

\begin{proof}
 By the elemental inequality:
\begin{equation}\label{eq9}
(a+b)^2\le a^2/(1-\vv)+b^2/\vv,\ \ \ a,b\in\R,\ \vv\in(0,1),
\end{equation}
for any $\gg>0$ and $t\ge 0$, we deduce from (A1) that
\begin{equation}\label{C2}
\begin{split}
\sup_{0\le s\le t}\Big(\e^{\gg s}\E|X(s)|^2\Big)
&\le\ff{1}{1-\vv}\sup_{0\le s\le t}\Big(\e^{\gg
s}\E|X(s)-G(X_s)|^2\Big)\\
&\quad+\ff{\kk^2}{\vv}\e^{\gg\tau}\sup_{-\tau\le s\le t}\Big(\e^{\gg
s}\E|X(s)|^2\Big).
\end{split}
\end{equation}
Due to $q>(1-\kk)^{-2}$ and $\kk\in(0,1)$, there exists $\gg<\ll$
sufficiently small such that
\begin{equation}\label{C10}
\kk\e^{\gg\tau/2}<1 \ \ \mbox{ and } \ \ \ff{\e^{\gg\tau
}}{(1-\kk\e^{\gg\tau/2})^2}<q.
\end{equation}
For $\gg>0$ sufficiently small  such that \eqref{C10} holds, if
\begin{equation}\label{C4}
\e^{\gg t}\E|X(t)-G(X_t)|^2\le\ff{\dd}{\ll}\e^{\gg
t}+(1+\kk)^2\|\xi\|_\8^2,\ \ \ t\ge0,
\end{equation}
 then \eqref{C2} gives that
\begin{equation*}
\begin{split}
\sup_{-\tau\le s\le t}\Big(\e^{\gg s}\E|X(s)|^2\Big)
&\le\ff{1}{1-\vv}\sup_{0\le s\le t}\Big(\ff{\dd}{\ll}\e^{\gg
s}+(1+\kk)^2\|\xi\|_\8^2\Big)+\ff{\kk^2}{\vv}\e^{\gg\tau}\sup_{-\tau\le
s\le t}\Big(\e^{\gg s}\E|X(s)|^2\Big).
\end{split}
\end{equation*}
Thus \eqref{C3} follows by taking $\vv=\kk\e^{\gg\tau/2}$. In what
follows, under (A1) and \eqref{eq1} we verify by a contradiction
argument that \eqref{C4} is indeed true for  sufficiently small
$\gg>0$ to be determined. Note from (A1) that
\begin{equation*}
|\xi(0)-G(\xi)|^2\le(1+\kk)^2\|\xi\|_\8^2.
\end{equation*}
If \eqref{C4} is not true, then there exist $\rho>0$ and
sufficiently small $h>0$ such that
\begin{equation}\label{C5}
\begin{split}
\e^{\gg t}\E|X(t)-G(X_t)|^2-\ff{\dd}{\ll}\e^{\gg t}&\le\e^{\gg
\rho}\E|X(\rho)-G(X_\rho)|^2-\ff{\dd}{\ll}\e^{\gg
\rho}\\
&=(1+\kk)^2\|\xi\|_\8^2,\ \ \ 0\le t\le\rho,
\end{split}
\end{equation}
however,
\begin{equation}\label{C6}
\begin{split}
\e^{\gg t}\E|X(t)-G(X_t)|^2-\ff{\dd}{\ll}\e^{\gg t}>\e^{\gg
t}\E|X(\rho)-G(X_\rho)|^2-\ff{\dd}{\ll}\e^{\gg \rho},\ \
\rho<t\le\rho+h.
\end{split}
\end{equation}
Taking \eqref{C3} and \eqref{C5} into account, we derive that
\begin{equation*}
\begin{split}
\E|X(t)|^2 &\le\ff{\dd/\ll+\e^{-\gg t}(\e^{\gg
\rho}\E|X(\rho)-G(X_\rho)|^2-\ff{\dd}{\ll}\e^{\gg
\rho})}{(1-\kk\e^{\gg\tau/2})^2},\ \ \ -\tau\le t\le \rho,
\end{split}
\end{equation*}
which, in particular, yields that
\begin{equation}\label{C9}
\begin{split}
\E|X(\rho+\theta)|^2 &\le \ff{\dd/\ll+\e^{-\gg
(\rho+\theta)}(\e^{\gg
\rho}\E|X(\rho)-G(X_\rho)|^2-\ff{\dd}{\ll}\e^{\gg
\rho})}{(1-\kk\e^{\gg\tau/2})^2}\\
&\le\ff{\e^{\gg\tau
}\E|X(\rho)-G(X_\rho)|^2}{(1-\kk\e^{\gg\tau/2})^2},\ \ \
-\tau\le\theta\le0,
\end{split}
\end{equation}
due to $\ff{\dd}{\ll}(1-\e^{-\gg\theta})\le0$ for
$-\tau\le\theta\le0.$ Thus we get from \eqref{eq2}, \eqref{C10} and
\eqref{C9}  that
\begin{equation}\label{C11}
\E\{2\<X(t)-G(X_t),b(X_t)\>+\|\si(X_t)\|^2\}\le\dd-\gg\E|X(t)-G(X_t)|^2,\
\ \ \rho\le t\le\rho+h
\end{equation}
by virtue of the continuity of sample path, where $h>0$ is
sufficiently small. Next, applying the It\^o formula and using
\eqref{C11} yields that
\begin{equation}\label{C12}
\begin{split}
\E(\e^{\gg(\rho+h)}|X(\rho+h)-G(X_{\rho+h})|^2)-\ff{\dd}{\ll}\e^{\gg(\rho+h)}
&\le\E(\e^{\gg\rho}|X(\rho)-G(X_{\rho})|^2)-\ff{\dd}{\ll}\e^{\gg
\rho}.
\end{split}
\end{equation}
Finally  we conclude that \eqref{C4} holds by the contradiction
between \eqref{C6} and \eqref{C12}.
\end{proof}

Our main result in this section is presented as below.

\begin{thm}\label{neutral}
{\rm Let (A1)-(A3) hold and $ \kk\in(0,1/2) $ and $
\aa_1>\aa_2/(1-2\kk)^2. $ Assume further that
\begin{equation}\label{eq13}
|G(\phi)-G(\psi)|\le\kk\|\phi-\psi\|_\8,\ \ \ \phi,\psi\in\C.
\end{equation}
 Then, \eqref{C1} has a unique
 invariant measure $\pi(\cdot)\in\mathcal {P}(\C)$ and is
 exponentially mixing. That is, there exists $\ll>0$ such
 that
\begin{equation*}
|P_tF(\xi)-\pi(F)|\le c\e^{-\ll t},\ \ \ t\ge0,\ F\in\B_b(\C),\
\xi\in\C.
\end{equation*}
}
\end{thm}

\begin{proof}
By Yuan et al. \cite[Theorem 3.2]{YZM}, if, for  a bounded subset
$U\subset\C$,
\begin{enumerate}
\item[($\N1$)] $\sup_{t\ge0}\sup_{\xi\in U}\E\|X_t(\xi)\|_\8^2<\8$;
\item[($\N2$)] $\lim_{t\to\8}\sup_{\xi,\eta\in
U}\E\|X_t(\xi)-X_t(\eta)\|_\8^2=0$,
\end{enumerate}
then $\P(t,\xi,\cdot)$ converges weakly to $\pi(\cdot)\in\mathcal
{P}(\C)$. For any $F\in C_b(\C)$   and $t,s\ge0$, by the Markovian
property of $\{X_t(\xi)\}_{t\ge0}$, one has
\begin{equation*}
P_{t+s}F(\xi)=P_sP_tF(\xi).
\end{equation*}
For fixed $t\ge0$, taking $s\to\8$ gives that
\begin{equation*}
\pi(F)=\pi(P_tF)
\end{equation*}
whenever $\P(t,\xi,\cdot)$ converges weakly to
$\pi(\cdot)\in\mathcal {P}(\C)$. Hence, \eqref{C1} admits an
invariant measure provided that $(\N1)$ and $(\N2)$ hold
respectively. In what follows, we claim that $(\N1)$ and $(\N2)$
hold under the conditions imposed. Following a similar argument to
that of \cite[Corollary 6.6, p.227]{M08} and taking Lemma \ref{Raz}
into consideration, we deduce that there exists $\gg>0$ sufficiently
small such that
\begin{equation}\label{eq4}
\sup_{t\ge-\tau}\E|X(t,\xi)|^2<\8 \ \mbox{ and }\
\E|X(t,\xi)-X(t,\eta)|^2\le c \e^{-\gg t},\ \ t\ge0.
\end{equation}
By the Burkhold-Davis-Gundy inequality, (A2)-(A3), for any
$t\ge2\tau$   we obtain that
\begin{equation}\label{eq5}
\begin{split}
\E\Big(\sup_{t-\tau\le s\le
t}|\Lambda(s,\xi)|^2\Big)&\le2\E(|\Lambda(t-\tau,\xi)|^2)+c\int_{t-\tau}^t\Big\{1+\sup_{-\tau\le
r\le s}\E|X(r,\xi)|^2\Big\}\d s
\end{split}
\end{equation}
with $\Lambda(t,\xi):=X(t,\xi)-G(X_t(\xi))$, and
\begin{equation}\label{eq6}
\begin{split}
\E\Big(\sup_{t-\tau\le s\le
t}|\Gamma(s,\xi,\eta)|^2\Big)&\le2\E(|\Gamma(t-\tau,\xi,\eta)|^2)+c\int_{t-\tau}^t\sup_{s-\tau\le
r\le s}\E|X(r,\xi)-X(r,\eta)|^2\d s,
\end{split}
\end{equation}
where $
\Gamma(t,\xi,\eta):=X(t,\xi)-X(t,\eta)-(G(X_t(\xi))-G(X_t(\eta))). $
Thus, the inequality \eqref{eq9}, (A1) and \eqref{eq4} give that
\begin{equation}\label{eq7}
\dd:=\sup_{t\ge-\tau}\E\Big(\sup_{t-\tau\le s\le
t}|\Lambda(s,\xi)|^2\Big)<\8
\end{equation}
and
\begin{equation}\label{eq8}
\E\Big(\sup_{t-\tau\le s\le t}|\Gamma(s,\xi,\eta)|^2\Big)\le
c\e^{-\gg t},\ \ \ t\ge2\tau.
\end{equation}
For any integer $n\ge2$, note from \eqref{eq9}, \eqref{eq13},  and
\eqref{eq7} that
\begin{equation*}
\begin{split}
\E\|X_{n\tau}(\xi)\|^2_\8&\le \ff{1}{\kk}\E\Big(\sup_{(n-1)\tau\le
s\le
n\tau}|G(X_s(\xi))|^2\Big)+\ff{\dd}{1-\kk}\\
&\le\kk\E\|X_{n\tau}(\xi)\|^2_\8+\kk\E\|X_{(n-1)\tau}(\xi)\|^2_\8+\ff{\dd}{1-\kk}.
\end{split}
\end{equation*}
By virtue of an induction argument,  due to $\kk\in(0,1/2)$, one
derive that
\begin{equation}\label{eq12}
\begin{split}
\E\|X_{n\tau}(\xi)\|^2_\8&\le\ff{\kk}{1-\kk}\E\|X_{(n-1)\tau}(\xi)\|^2_\8+\ff{\dd}{(1-\kk)^2}\\
&\le\Big(\ff{\kk}{1-\kk}\Big)^n\|\xi\|_\8^2+\ff{\dd}{(1-\kk)^2}\Big\{1+\ff{\kk}{1-\kk}+\cdots+\Big(\ff{\kk}{1-\kk}\Big)^{n-1}\Big\}\\
&\le\|\xi\|_\8^2+\ff{\dd}{(1-\kk)(1-2\kk)}.
\end{split}
\end{equation}
Observe that for any $t\ge0$ there exists an $n\ge0$ such that
$t\in[n\tau,(n+1)\tau)$ and
\begin{equation*}
\E\|X_t(\xi)\|^2_\8\le\E\|X_{n+1}(\xi)\|^2_\8+\E\|X_n(\xi)\|^2_\8.
\end{equation*}
Then $(\N1)$ follows immediately from \eqref{eq12}. On the other
hand, by  \eqref{eq9}, \eqref{eq13} and \eqref{eq8}, for any integer
$n\ge2$, it follows that
\begin{equation*}
\begin{split}
\E\|X_{n\tau}(\xi)-X_{n\tau}(\eta)\|^2_\8&\le
\ff{1}{\kk}\E\Big(\sup_{(n-1)\tau\le s\le
n\tau}|G(X_s(\xi))-G(X_s(\eta))|^2\Big)+\ff{c\e^{-n\gg \tau}}{1-\kk}\\
&\le\kk\E\|X_{n\tau}(\xi)-X_{n\tau}(\eta)\|^2_\8+\kk\E\|X_{(n-1)\tau}(\xi)-X_{(n-1)\tau}(\eta)\|^2_\8+\ff{c\e^{-n\gg
\tau}}{1-\kk}.
\end{split}
\end{equation*}
Also by an induction argument, we obtain that
\begin{equation}\label{eq11}
\begin{split}
\E\|X_{n\tau}(\xi)-X_{n\tau}(\eta)\|^2_\8 &\le
\Big(\ff{\kk}{1-\kk}\Big)^n\|\xi-\eta\|^2_\8+\ff{c}{(1-\kk^2)}\Big\{\Big(\ff{\kk}{1-\kk}\Big)^{n-1}\e^{-\gg
\tau}\\
&\quad+\Big(\ff{\kk}{1-\kk}\Big)^{n-2}\e^{-2\gg
\tau}+\cdots+\e^{-n\gg \tau}\Big\}\\
&\le\Big(\ff{\kk}{1-\kk}\Big)^n\|\xi-\eta\|^2_\8+\ff{\e^{-n\gg\tau}(1-q^n)}{1-q}\\
&\le\|\xi-\eta\|^2_\8\e^{-pn\gg\tau}+\ff{\e^{-n\gg\tau}}{1-q}\\
&\le c\e^{-(p\wedge1)n\gg\tau},
\end{split}
\end{equation}
where
\begin{equation*}
p:=\ff{1}{\gg\tau}\log\Big(\ff{1-\kk}{\kk}\Big) \mbox{ and }
q:=\kk\e^{\gg\tau}/(1-\kk)<1
\end{equation*}
 for $\kk\in(0,1/2)$ and $\gg>0$
sufficiently small. Next,  for any $t>0$,   notice that there exists
$n\ge0$ such that $t\in[n\tau,(n+1)\tau)$ and by \eqref{eq11} that
\begin{equation*}
\begin{split}
\E\|X_t(\xi)-X_t(\eta)\|^2_\8&\le\E\|X_{n+1}(\xi)-X_{n+1}(\eta)\|^2_\8+\E\|X_n(\xi)-X_n(\eta)\|^2_\8\\
&\le
c\e^{-(p\wedge1)(n+1)\gg\tau}+c\e^{(p\wedge1)\gg\tau}\e^{-(p\wedge1)(n+1)\gg\tau}\\
&\le c\e^{-(p\wedge1)\gg t}.
\end{split}
\end{equation*}
Consequently, $(\N2)$ holds.  Finally the desired assertion follows
by repeating the latter proof of Theorem \ref{Erg}.
\end{proof}

\begin{rem}
{\rm There are some examples such that (A1) and \eqref{eq13} hold,
e.g., $G(\phi)=\kk\phi(-\tau)$,  $G(\phi)=\kk\phi(-\tau(t))$ and
$G(\phi)=\kk\int_{-\tau}^0\phi(\theta)\mu(\d\theta)$ for $\phi\in\C$
and  $\kk\in(0,1/2)$, where $\mu(\cdot)$ is a probability measure on
$[-\tau,0].$ }
\end{rem}

\begin{rem}
{\rm By a generalized  Razumikhin-type theorem, we give a uniform
bound of segment processes associated with neutral SDEs, while the
method adopted in Lemma \ref{uni} seems hard to work  because of the
appearance of neutral term. Moreover, the Arzel\`{a}--Ascoli method
utilized in Theorem \ref{Erg} does not apply to neutral SDEs either
although it can deal with the {\it non-autonomous} retarded SDEs.
The trick applied in Theorem \ref{neutral} is call a
``stability-in-distribution approach'' and our main result, Theorem
\ref{neutral}, includes neutral SDEs with
constant/variable/distributed time-lags.

}
\end{rem}

\begin{rem}
{\rm Let $(H,\<\cdot,\cdot\>_H,\|\cdot\|_H)$ be a real separable
Hilbert space and $V$ a Banach space  such that $V\hookrightarrow H$
continuously and densely.  Via the Riesz isomorphism,
\begin{equation*}
V\hookrightarrow H\equiv H^*\hookrightarrow V^*,
\end{equation*}
where $H^*$ and $V^*$ are the dual space of $H$ and $V$
respectively. Stability-in-distribution approach can be applied to a
non-linear retarded Stochastic Partial Differential Equation (SPDE)
on the Gelfand triple $(V,H,V^*)$
\begin{equation}\label{02}
\d X(t)=\{A(X(t))+b(X_t)\}\d t+\si(X_t)\d W(t),
\end{equation}
where $A:V\mapsto V^\prime$ is  a family of nonlinear monotone and
coercive operators. However, the Arzel\`{a}--Ascoli approach seems
hard to apply to \eqref{02} because of the monotone and coercive
property of $A.$ Therefore, the the Arzel\`{a}--Ascoli method and
the stability-in-distribution approach possess their respective
advantages. }
\end{rem}

\section{Exponential Mixing for Retarded SDEs with
Jumps}\label{sec:4} In the last two sections, we investigate the
ergodic property of retarded SDEs with continuous sample paths under
the uniform topology. In this section, we turn to the case of
retarded SDEs with discontinuous paths.

We further need to introduce some additional notation and notions.
Let $\D:=D([-\tau, 0];\R^n)$ denote the collection of all
c\`{a}dl\`{a}g paths $f: [-\tau,0]\mapsto\R^n$. Recall that a path
$f: [-\tau,0]\mapsto\R^n$ is called c\'{a}dl\'{a}g if it is
right-continuous having finite left-hand limits. Let $\LL$ denote
the class of   increasing homeomorphisms,  and
\begin{equation*}
\|\ll\|^\circ:=\sup_{-\tau\le
s<t\le0}\Big|\log\ff{\ll(t)-\ll(s)}{t-s}\Big|<\8.
\end{equation*}
Under the uniform metric
$\|\zeta\|_\8:=\sup_{-\tau\le\theta\le0}|\zeta(\theta)|$ for each
$\zeta\in\D$, the space $\D$ is complete but not separable. For any
$\xi,\eta\in\D$, define the Skorhod metric $\d_S$ on $\D$ by
\begin{equation}\label{03}
\d_S(\xi,\eta):=\inf_{\ll\in\LL}\{\|\ll\|^\circ\vee\|\xi-\eta\circ\ll\|_\8\},
\end{equation}
where $\eta\circ\ll$ means the composition of mappings $\eta$ and
$\ll$. Under the skorohod metric $\d_S$,  $\D$ is not only complete
but also separable (see e.g. \cite[Theorem 12.2, p.128]{B}). For the
space $\D$, the uniform metric $\|\cdot\|_\8$ may lead to certain
misinterpretation of the actual situation while the Skorohod metric
$\d_S$ in the mathematical modeling of a certain processes gives a
more accurate representation of the processes, allows researches to
perform a correct analysis of a real situation, and make a credible
forecast of the possible outcomes of similar processes. For more
details on   the Skorohod Metric, we refer to \cite[Chapter 4]{B}.
Let $(\Y,\B(\Y),m(\cdot))$ be a measurable space, $D_p$ a countable
subset of $\R_+$ and $p:D_p\mapsto \Y$ an adapted process taking
value in $\Y$. Then, as in Ikeda and Watanabe \cite[p.59]{IW89}, the
Poisson random measure $N(\cdot,\cdot):\B(\R_+\times\Y)\times
\OO\mapsto\mathbb{N}\cup\{0\}$,  defined on the complete filtered
probability space $(\OO,\F,\{\F_t\}_{t\ge0},\P)$,  can be
represented by
\begin{equation*}
N((0,t]\times\GG)=\sum_{s\in D_p,s\le t}{\bf1}_\GG(p(s)),\ \ \
\GG\in\B(\Y).
\end{equation*}
In this case, we say that $p$ is a Poisson point process and $N$ is
a Poisson random measure. Let $m(\cdot):=\E N((0,1]\times\cdot)$.
Then, the compensated Poisson random measure
\begin{equation*}
\tt N(\d t,\d z):=N(\d t,\d z)-\d t  m(\d z) \ \ \mbox{ is a
martingale}.
\end{equation*}
A stochastically continuous Markovian semigroup $P_t$ is called
eventually Feller if $P_tf\in C_b(\C)$ for any $F\in C_b(\C)$ and
$t\ge t_0,$ where $t_0\ge0$ is some constant, and immediately Feller
for $t_0=0$.

Consider a non-autonomous retarded SDE with jump
\begin{equation}\label{b1}
\d X(t)=b(t,X_t)\d t+\int_\GG\si(t,X_{t-},z)\tt N(\d t,\d z),\ \ \
t\ge0
\end{equation}
with the initial value $\xi\in\D$ which is independent of
$N(\cdot,\cdot)$, where
$X_{t-}(\theta):=X((t+\theta)-):=\lim_{s\uparrow t+\theta}X(s)$ for
$\theta\in[-\tau,0]$, $b:[0,\8)\times\D\times\OO\mapsto\R^n$ and
$\si:[0,\8)\times\D\times\OO\mapsto\R^n\times\GG\mapsto\R^n$ are
progressively measurable.

For any $\phi,\psi\in\D$ and any $t\ge0$, we assume that
\begin{enumerate}
\item[(B$1$)] There exist $\aa_1>\aa_2>0$ such that
\begin{equation*}
\begin{split}
&\E\Big\{2\<\phi(0)-\psi(0),b(t,\phi)-b(t,
\psi)\>+\int_\GG|\si(t,\phi,z)-\si(t,\psi,z)|^2m(\d z)\Big\}\\
&\quad\le-\aa_1\E|\phi(0)-\psi(0)|^2+\aa_2\sup_{-\tau\le\theta\le0}\E|\phi(\theta)-\psi(\theta)|^2;
\end{split}
\end{equation*}
\item[(B$2$)] There exists $\aa_3>0$ such that
\begin{equation*}
\E|b(t,\phi)-b(t,
\psi)|^2+\E\int_\GG|\si(t,\phi,z)-\si(t,\psi,z)|^2m(\d z)\le
\aa_3\sup_{-\tau\le\theta\le0}\E|\phi(\theta)-\psi(\theta)|^2.
\end{equation*}
\end{enumerate}

The main result in this section is stated as follows.

\begin{thm}\label{jump}
{\rm Under (B1)-(B2), \eqref{b1} has a unique
 invariant measure $\pi(\cdot)\in\mathcal {P}(\D)$ and is
 exponentially mixing. More precisely, there exists $\ll>0$ such
 that
\begin{equation*}
|P_tF(\xi)-\pi(F)|\le c\e^{-\ll t},\ \ \ t\ge\tau,\ F\in\B_b(\D),\
\xi\in\D.
\end{equation*}
}
\end{thm}

\begin{proof}
The whole proof is divided into the following three steps.

\smallskip
 \noindent{\bf
Step 1:} Claim a uniform bound of $X_t$:
\begin{equation}\label{c4}
\sup_{t\ge-\tau}\E\|X_t\|_\8^2<\8.
\end{equation}
Following a similar argument to derive \eqref{A8}, we derive that
\begin{equation}\label{c1}
\dd:=\sup_{t\ge-\tau}\E|X(t)|^2<\8.
\end{equation}
By the It\^o formula, for any $t\ge\tau$ and $\theta\in[-\tau,0]$,
it follows that
\begin{equation}\label{c3}
\begin{split}
|X(t+\theta)|^2
&=|X(t-\tau)|^2+2\int_{t-\tau}^{t+\theta}\langle X(s), b(s,X_s)\rangle\d s\\
&\quad+\int_{t-\tau}^{t+\theta}\int_\GG|\si(s,X_{s-},z)|^2 N(\d s,\d
z)+2\Pi(t,t+\theta),
\end{split}
\end{equation}
in which
\begin{equation*}
\Pi(t,t+\theta):=\int_{t-\tau}^{t+\theta}\int_\GG\langle X(s-),
\si(s,X_{s-},z)\rangle\tt N(\d s,\d z).
\end{equation*}
Next, due to  the Burkhold-Davis-Gundy inequality (see e.g.
\cite[Theorem 48, p.193]{P04}), and the Jensen inequality, we derive
that
\begin{equation}\label{c2}
\begin{split}
&\E\Big(\sup_{-\tau\le \theta\le 0}|\Pi(t,t+\theta)|\Big)\le c\E\sqrt{[\Pi,\Pi]_{[t-\tau,t]}}\\
&\le c\E\ss{\int_{t-\tau}^t\int_\GG |\< X(s-),
\si(s,X_{s-},z)\>|^2N(\d s,\d z)}\\
&\le c\ss{\E\|X_t\|^2_\8\E\int_{t-\tau}^t\int_\GG
|\si(s,X_{s-},z)|^2N(\d s,\d z)}\\
&\le\ff{1}{4} \E\|X_t\|^2_\8+c\E\int_{t-\tau}^t\int_\GG
|\si(s,X_s,z)|^2m(\d z)\d s,
\end{split}
\end{equation}
where $[\Pi,\Pi]_{[t-\tau,t]}$ stands for the quadratic variation
process (square bracket process) of $\Pi(t,t-\tau)$. Taking
\eqref{c3} and \eqref{c2} into consideration and using (B1) and
(B2), we arrive at
\begin{equation*}
\begin{split}
\E\|X_t\|^2
&\le2\E|X(t-\tau)|^2+c\int_{t-\tau}^t\Big(1+\sup_{-\tau\le r\le
s}\E|X(r)|^2\Big)\d s,\ \ \ t\ge\tau.
\end{split}
\end{equation*}
This, together with \eqref{c1}, leads to  \eqref{c4}.

\smallskip
\smallskip
\noindent {\bf Step 2:} Existence of an invariant measure. For
$\theta\in[-\tau,0]$ and $\wdt \theta\in[0,\3]$, where $\3>0$ is an
arbitrary constant such that $\theta+\3\in[-\tau,0].$ Set
$\E_s\cdot:=\E(\cdot|\F_s),s\ge0$. By the It\^o isometry, for any
$t\ge\tau$, we obtain from \eqref{b1} that
\begin{equation*}
\begin{split}
\E_{t+\theta}|X_t(\theta+\wdt \theta)-X_t(\theta)|^2 & =\E_{t+\theta}|X(t+\theta+\wdt \theta)-X(t+\theta)|^2\\
&\le
c\int_{t+\theta}^{t+\theta+\3}\E_{t+\theta}\Big\{|b(s,X_s)|^2+\int_\GG|\si(s,X_{s-},z)|^2m(\d
z)\Big\}\d s.
\end{split}
\end{equation*}
By virtue of (B1)-(B2) and \eqref{c4}, there is a $\gamma_0 (t,\3)$ satisfying
$$\E_{t+\theta}|X(t+\theta+\wdt \theta)-X(t+\theta)|^2 \le \E_{t+\theta}\gamma_0(t,\3).$$
Taking expectation and $\limsup_{t\to \infty}$ followed by
$\lim_{\3\to 0}$, we obtain from (B1)-(B2) and \eqref{c4} that
\begin{equation}\label{c5}
\lim_{\3\to0}\limsup_{t\to \infty}\E\gamma_0(t,\3)=0.
\end{equation}
In view of \eqref{c4} and \eqref{c5}, combining with \cite[Theorem
3, p.47]{K}, we conclude that $X_t$ is tight under the Skorohod
metric $\d_S$. For each integer $n\ge1$, set
\begin{equation*}
\mu_n(\cdot):=\ff{1}{n}\int_0^n\P_t(\xi,\cdot)\d t,
\end{equation*}
where $\P_t(\xi,\cdot)$ is the Markovian transition kernel of
$X_t(\xi)$. Since  $X_t$ is tight under the Skorohod metric $\d_S$,
for any $\vv>0$ there exists a compact subset $U\in\B(\D)$ such that
$\P(X_t\in U)\le1-\vv$. Hence we have $\mu_n(U)\le1-\vv$. That is,
$\{\mu_n(\cdot)\}_{n\ge1}$ is tight. Observe from Rei$\beta$ et al.
\cite{RR} that $P_t$ is eventually Feller. As a result, by the
Krylov-Bogoliubov theorem \cite[Theorem 3.1.1, p.21]{DZ}, we
conclude that \eqref{b1} has a unique
 invariant measure $\pi(\cdot)\in\mathcal {P}(\D)$, where $\D$ is
 equipped with the Skorohod topology.

\smallskip
\smallskip
\noindent {\bf Step 3:} Exponential Mixing. By  the Halanay-type
inequality \ref{Halanay}, one has from (B1) that
\begin{equation}\label{c7}
\E|X(t,\xi)-X(t,\eta)|\le c\e^{-\ll t}
\end{equation}
for some $\ll>0.$ Carrying out a similar argument to get \eqref{B5},
we derive from \eqref{c7} that
\begin{equation}\label{c6}
\E\|X_t(\xi)-X_t(\eta)\|_\8\le c\e^{-\ll t}.
\end{equation}
By the definition of $\d_S$,  note that
\begin{equation}\label{c8}
\d_S(\xi,\eta)\le\|\xi-\eta\|_\8
\end{equation}
by choosing $\ll(t)\equiv t$ in \eqref{03}. For any bounded
Lipschitz function $F:\D\to\R$, by the invariance of
$\pi(\cdot)\in\mathcal {P}(\D)$, it follows from \eqref{c8} that
\begin{equation*}
\begin{split}
 |P_tF(\xi)-\pi(F)|&\le\int_\C| P_tF(\xi)-P_tF(\eta)|\pi(\d\eta)\le
 c\E\d_S(X_t(\xi),X_t(\eta))\\
 &\le c\E\|X_t(\xi)-X_t(\eta)\|_\8.
 \end{split}
\end{equation*}
Consequently,  the desired assertion follows from \eqref{c6} and a
monotone class argument for any $F\in\B_b(\D)$.
\end{proof}

\begin{rem}
{\rm Since the semigroup $P_t$ generated by the segment process
associated with \eqref{b1} is not stochastically continuous for
$t\in[0,\tau)$ (see e.g. \cite[p.1416]{RR}), $P_t$ is not
immediately Feller. Hence, $t\ge\tau$ imposed in Theorem \ref{jump}
is natural. However, Theorem \ref{Erg} and Theorem \ref{neutral}
hold respectively for any $t\ge0$. This further shows the different
features of  retarded SDEs with continuous sample paths and the ones
driven by jump processes.

}
\end{rem}

\begin{rem}
{\rm By a remote start method (or dissipative method), Bao et al.
\cite{BYY} discussed ergodic property for several class of
functional SDEs, which cannot cover the equations considered in this
paper. In particular,  the remote start method applied therein only
deals with the autonomous functional SDEs, while the approaches
adopted in Theorem \ref{Erg} and Theorem \ref{jump} even work for
non-autonomous cases. }
\end{rem}

\begin{rem}
{\rm Since the space $\D$ is complete but not separable  under the
uniform metric, the Arzel\`{a}-Ascoli method adopted in Theorem
\ref{Erg} is unavailable for functional SDEs with jumps. Moreover,
  Kurtz's criterion used in Theorem \ref{jump} also applies to
infinite-dimensional semi-linear retarded SPDEs driven by jump
processes. Hence our method is dimensional-free while the trick used
in Bo and Yuan \cite{BY} is dimensional-dependent. }
\end{rem}

\end{document}